\documentclass[a4paper,10pt,reqno]{amsart}
\usepackage{amsfonts, amssymb, amsmath, wrapfig, amsthm, palatino,
  euler, microtype, enumitem}
\usepackage{longtable}
\usepackage[all]{xy}
\usepackage{epsfig}

\def\hat{\widehat}
\def\a{\alpha}
\def\b{\beta}
\def\g{\gamma}
\def\d{\delta}
\def\l{\lambda}

\def\s{\sigma}
\def\ph{\varphi}

\def\A{\operatorname{A}}
\def\D{\operatorname{D}}
\def\E{\operatorname{E}}

\def\Int{{\mathbb Z}}

\def\la{\langle}
\def\ra{\rangle}

\newtheorem{theorem}{Theorem}
\newtheorem{lemma}{Lemma}
\newtheorem{proposition}{Proposition}
\newtheorem{definition}{Definition}
\newtheorem{corollary}{Corollary}

\title[Equations in the adjoint representation]
{Equations determining the orbit of the highest weight vector in the
  adjoint representation}
\author{Alexander Luzgarev}
\address{Saint Petersburg State University}
\date{January 4, 2014}
\thanks{This work was started in the framework of the DAAD program
  ``Mikhail Lomonosov''. Later it was supported by the Russian
  Foundation for Basic Research (project nos. 12-01-31100, 13-01-00429,
  13-01-91150, 13-01-92699, and 14-01-31515), the Golda Meir
  Postdoctoral Fellowship, and by the State Financed research task
  6.38.191.2014.}

\begin{document}

\begin{abstract}
We explicitly construct a set of quadratic equations defining the
highest weight vector orbit for adjoint representations of
Chevalley groups of types $\D_l$, $\E_6$, $\E_7$, and $\E_8$.
The combinatorics of these equations is related to the combinatorics
of embeddings of the root system of type $\A_3$.  We believe that the
constructed equations provide a prominent framework for calculations
with exceptional groups in adjoint representations, which is
particularly interesting for groups of type $\E_8$.
\end{abstract}

\maketitle

\section{Introduction}

The highest weight vector orbit in an irreducible representation of a
Chevalley group over an algebraically closed field is an intersection
of quadrics (cf.~\cite{Lichtenstein}).
We explicitly describe a set of quadratic equaions on this orbit over
an arbitrary commutative ring.
First of all, we have {\it square equations} described by
Nikolai Vavilov in~\cite{Vavilov_numerology_translation} for
microweight
representations as well as for adjoint ones. In some microweight cases
those equations exhaust all equations defining the highest weight
orbit (over an algebraically closed field). In the adjoint cases
square equations are clearly not enough: for example (as Vavilov
pointed out in~\cite{Vavilov_numerology_translation}), they do not
contain
coordinates corresponding to the zero weight.
We cannot get on with an $\A_2$-proof of the structure theorems for
$\E_8$ (cf.~\cite{Vavilov_Luzgarev_E8_translation}) without zero weight coordinates.

The equations on the highest weight vector orbit in the adjoint
representation of a group of type $\A_l$ are well known: they are
called {\it Pl\"ucker equations}. On the other hand, non-simply-laced
root systems are generally a little harder to deal with. That is why
we consider only the remaining simply-laced root systems $\D_l$,
$\E_6$, $\E_7$, and $\E_8$. Moreover, in order to evade some
difficulties relating to triality in $\D_4$ we take $l\geq 5$ in the
$\D_l$ case. In any way, we include $\D_l$ only because our
constructions work verbatim in this case; our main goal is to obtain
equations for exceptional groups.

We construct, in addition to the aforementioned square equations, two
more classes of equations; all of them contain zero-weight
coordinates. The combinatorics of these equations is also intimately
related to the ``numerology of maximal squares'' studied
in~\cite{Vavilov_numerology_translation}
and~\cite{Vavilov_more_numerology_translation}.
The same equations are produced in a more general context (but in
slightly less explicit form) by Victor Petrov, Nikolai
Vavilov, and myself in~\cite{Luzgarev_Petrov_Vavilov}.

The basic calculations that lead to the present paper were performed
by the author in 2007--2008 at the Universit\"at Bielefeld.
The author thanks Anthony Bak for his hospitality and support, and
Nikolai Vavilov for extremely useful discussions.

\section{The equations}

Everywhere in this paper $\Phi=\D_l$, $l\geq 5$ or $\Phi=\E_l$,
$l=6,7,8$.
Let $\{\a_1,\dots,\a_l\}=\Pi\subset\Phi$ be a fundamental system in
$\Phi$ (its elements will be called fundamental roots). Our numbering
of fundamental roots always follows Bourbaki~\cite{Bourbaki46_french}.
For $\a\in\Phi$
we set $\a=\sum_{s=1}^lm_s(\a)\a_s$.

Let $G=G(\Phi,R)$ be the simply connected Chevalley group of type
$\Phi$ over a commutative ring $R$ with $1$.
We work with the adjoint representation of $G(\Phi,R)$, which gives us
the irreducible action of $G(\Phi,R)$ on a free $R$-module $V$ of rank
$l(2l-1)$, $78$, $133$, $248$ for $\Phi=\D_l$, $\E_6$, $\E_7$, $\E_8$
respectively.
By $\Lambda$ we denote the set of weights of our representation
{\it with multiplicities}. More precisely,
$\Lambda=\Lambda^*\sqcup\Delta$, where $\Lambda^*=\Phi$ is the set of
non-zero weights, and $\Delta=\{0_1,\dots,0_l\}$ is the set of zero
weights. We fix an admissible base $e^\lambda$, $\lambda\in\Lambda$ in
$V$. Hence we have the vectors $e^\a$ for $\a\in\Phi$ and
$\hat{e}^i = e^{0_i}$ for $i=1,\dots,l$.
Then a vector $v\in V$ can be uniquely written as
$v = \sum_{\lambda\in\Lambda} v_\lambda e^\lambda =
\sum_{\alpha\in\Phi} v_\alpha e^\alpha +
\sum_{i=1}^l \hat{v}_i\hat{e}^i$. We will often simply write
$v = (v_\lambda)$.

The root system $\Phi$ is a subset of a Euclidean space $E$ with the
scalar product denoted by $(\cdot,\cdot)$.
We will also use a bilinear product defined by
$\la\a,\b\ra = 2(\a,\b)/(\b,\b)$ for $\a,\b\in E$ (for
$\a,\b\in\Phi$ we get the {\it Cartan numbers}). Note that our root
system $\Phi$ is simply-laced, which means that all roots have length
$1$; therefore $\la\a,\b\ra = 2(\a,\b)$ for $\a,\b\in\Phi$.
We denote by $\angle(\a,\b)$ the angle between $\a,\b\in E$.
Note that for $\a,\b\in\Phi$ the scalar product $(\a,\b)$ is $0, 1/2,
-1/2, 1, -1$ if $\a\perp\b$, $\a-\b\in\Phi$, $\a+\b\in\Phi$, $\a=\b$,
$\a=-\b$ respectively.

The structure constants $N_{\alpha,\beta}$, $\alpha,\beta\in\Phi$ of
the simple complex Lie algebra of type $\Phi$ are described in detail
in~\cite[\S~1]{Vavilov_znaki_translation}. We often use the
identities for structure constants summarized there without any
explicit reference. Note that in our case always $N_{\a,\b}=0$ or $\pm 1$.

\par\smallskip
$\bullet$ {\bf The $\pi/2$-equations.}
Suppose $\a,\b\in\Phi$ and
$\angle(\a,\b)=\pi/2$.
Let us look at all other (unordered) pairs of roots with the same sum:
$$
S_{\pi/2}(\a,\b)=\{\{\g,\d\}\mid \g+\d=\a+\b,\{\g,\d\}\neq\{\a,\b\}\}.
$$
Consider the foolowing equation on a vector
$v=(v_\l)_{\l\in\Lambda}\in V$:
\begin{equation}\label{pi/2-equation}
v_\a v_\b=\sum_{\{\g,\d\}\in S_{\pi/2}(\a,\b)}N_{\a,-\g}N_{\b,-\d}v_\g v_\d.
\end{equation}
We will call it {\it the $\pi/2$-equation} corresponding to the pair
$\{\a,\b\}$. First of all, we need to prove that the right hand side
makes sense: we could swap $\gamma$ with $\delta$ and get a
different-looking coefficient.
But it follows from the identity $(C5)$ in~\cite{Vavilov_znaki_translation} that
$N_{\a,-\g}N_{\b,-\d}=N_{\a,-\d}N_{\b,-\g}$.
Next, note that $(\a,\g)+(\a,\d)=(\a,\g+\d)=(\a,\a+\b)=(\a,\a)=1$,
while $\a\neq\g$, $\a\neq\d$. 
Therefore $(\a,\g)=(\a,\d)=1/2$. It follows that
$\angle(\a,\g)=\angle(\a,\d)=\pi/3$.

For the rest of the paper, put $k=l,4,5,7$ for
$\Phi=\D_l,\E_6,\E_7,\E_8$, respectively.
In order to write the $\pi/2$-equation in a more symmetric form,
recall a definition from~\cite{Vavilov_numerology_translation}.
\begin{definition}\label{def:maximal_square}
A set of roots $\{\b_i\}$, $i=1,\dots,k,-k,\dots,-1$
such that $\angle(\b_i,\b_{-i})=\pi/2$ for every $i=1,\dots,k$, and
$\angle(\b_i,\b_j)=\pi/3$ for $i\neq\pm j$,
is called a {\bf maximal square}.
\end{definition}
For a maximal square $\{\b_i\}$ the sum $\b_i+\b_{-i}$ does not depend
on $i$.
Therefore the set of roots contained in the pairs from
$S_{\pi/2}(\a,\b)$, together with the roots $\a$ and $\b$,
is a maximal square (this was proved in~\cite[Theorem
1]{Vavilov_numerology_translation}).
We shall prove shortly that our $\pi/2$-equation is uniquely
determined by this maximal square, and does not depend on the choice
of an orthogonal pair of roots $\{\a,\b\}$.
Let us fix an index $j=1,\dots,-1$.
If we put $\b_1=\a$, $\b_{-1}=\b$, and
$S_{\pi/2}(\a,\b) = \{\{\b_i,\b_{-i}\}\mid i=2,\dots,k\}$,
the $\pi/2$-equation can be rewritten as
$$
v_{\b_1}v_{\b_{-1}} =
\sum_{i\geq 2}N_{\b_1,-\b_i}N_{\b_{-1},-\b_{-i}}v_{\b_i}v_{\b_{-i}}.
$$
The sign column $c(j)\in(\Int/2\Int)^{2k}$ is defined as follows.
$$
c(j)_i=
\begin{cases} 1,&\text{if $i=\pm j$},\\
-N_{\b_j,-\b_i}N_{\b_{-j},-\b_{-i}},&\text{if $i\neq\pm j$}.\\
\end{cases}
$$
Another equivalent form of the $\pi/2$-equation is
$$
\sum_{i=1}^k c(1)_iv_{\b_i}v_{\b_{-i}}=0.
$$

The following lemma says that if we take another orthogonal pair in
$S_{\pi/2}(\a,\b)$ instead of $\{\a,\b\}$, we will get the same equation.

\begin{lemma}
For any $j,h=1,\dots,-1$ we have
$$
c(h)-c(h)_jc(j)=0.
$$
\end{lemma}
\begin{proof}
Immediately follows from~\cite[Theorem 3]{Vavilov_numerology_translation}.
\end{proof}

\par\smallskip
$\bullet$ {\bf The $2\pi/3$-equations.}

Suppose again that $\a,\b\in\Phi$ and $\angle(\a,\b)=\pi/2$.
Consider all pairs of roots $\{\g,\d\}$ such that $\g+\d=\a$
and $\g,\d$ are not orthogonal to $\b$.
Note that if $\g\perp\b$ and $\g+\d=\a$, then $(\d,\b) =
(\a-\g,\b)=0$, so $\d\perp\b$.
Also, $0=(\a,\b)=(\g+\d,\b)=(\g,\b)+(\d,\b)$. Therefore
for such a pair $\{\g,\d\}$ one of the angles $\angle(\g,\b)$,
$\angle(\d,\b)$ is $2\pi/3$, while the other is $\pi/3$.
Put
$$
S_{2\pi/3}(\a,\b)=\{\{\gamma,\delta\}\mid \gamma+\delta=\a,(\g,\b)\neq 0\}.
$$

Consider the following equation on a vector
$v=(v_\l)_{\l\in\Lambda}\in V$:
\begin{equation}\label{2pi/3-equation}
v_\a\cdot\sum_{s=1}^l\langle\b,\a_s\rangle \hat v_{s}
=-\sum_{\substack{\{\g,\d\}\in S_{2\pi/3}(\a,\b),\\\angle(\g,\b)=\pi/3}}N_{\g,\d}v_\g v_\d.
\end{equation}
We will call it {\it the $2\pi/3$-equation} corresponding to the pair
$(\a,\b)$. 

The pairs in $S_{2\pi/3}(\a,\b)$ are related to the embeddings of root
systems $\A_3\subset\Phi$.
In order to see that, consider a pair $\{\g,\d\}\in S_{2\pi/3}(\a,\b)$.
We may assume that $(\g,\b)=1/2$, $(\d,\b)=-1/2$.
Then the roots $\d,\g,\b-\g$ form a fundamental system of a root
subsystem $\Psi\subseteq\Phi$ of type $\A_3$.
We can write the roots $\a,\b$ in the Dynkin notation for this
fundamental system as follows: $\a=110$, $\b=011$.
Note that $\Psi$ contains $\g',\d'$ for exactly one more pair
$\{\g',\d'\}\in S_{2\pi/3}(\a,\b)$,
namely, the pair $\{\g',\d'\}= \{111,-001\}$. In other words,
$\g'=\d+\b$, $\d'=\g-\b$.
The pairs $\{\g,\d\}$ and $\{\g',\d'\}$ are said to be{\it conjugate}.

Note that $|S_{2\pi/3}(\a,\b)|=2(l-1),6,8,12$ for
$\Phi=\D_l,\E_6,\E_7,\E_8$, respectively.
We see that the number of conjugate pairs in $S_{2\pi/3}(\a,\b)$
is one less than the number of pairs of orthogonal roots in a maximal
square.
This is not a coincidence:
if we fix an orthogonal pair $(\a,\b)$ in a maximal square and take
any of the remaining pairs, together they span a root subsystem of
type $\A_3$. There are exactly $k-1$ of these subsystems, and each
contains exactly two of conjugate pairs from $S_{2\pi/3}(\a,\b)$.

We get the following equivalent description of $S_{2\pi/3}(\a,\b)$:
\begin{lemma}\label{2pi/3-equiv}
Suppose that $\a,\b\in\Phi$, $\a\perp\b$. Let
$\Omega=\{\b_1,\dots,\b_{-1}\}$ be a maximal square such that
$\b_1=\a$, $\b_{-1}=\b$, and $\b_i\perp\b_{-i}$ for every $i$.
Then
$$
S_{2\pi/3}(\a,\b)=\{\{\b_1-\b_i,\b_i\}\mid i=2,\dots,-2\},
$$
and the $2\pi/3$-equation corresponding to the pair $(\a,\b)$
can be rewritten as follows:
$$
v_{\b_1}\cdot\sum_{s=1}^l\langle\b_{-1},\a_s\rangle \hat v_{s} =\sum_{i\neq\pm 1}N_{\b_1,-\b_i}v_{\b_1-\b_i}v_{\b_i}.
$$
\end{lemma}
\begin{proof}
Note that $(\b_1-\b_i)+\b_i=\b_1=\a$ and $(\b_i,\b)=(\b_i,\b_{-1})\neq
0$ for $i=2,\dots,-2$.
This means that all pairs $\{\b_1-\b_i,\b_i\}$ are in
$S_{2\pi/3}(\a,\b)$.
In order to prove the reverse inclusion, consider a pair $\{\g,\d\}\in
S_{2\pi/3}(\a,\b)$.
We may assume that $\angle(\g,\b)=2\pi/3$.
Then the roots $\a-\g$, $\g+\b$ are orthogonal, and their sum is
$\a+\b$; it follows that $\a-\g,\g+\b\in\Omega$, so that
$\d=\a-\g=\b_i$ for some $i$. It remains to note that
$N_{\b_1-\b_i,\b_i}=N_{\b_i,-\b_1}=N_{\b_1,-\b_i}$ by the identities
$(C4)$ and $(C1)$ from~\cite{Vavilov_znaki_translation}
\end{proof}

\par\smallskip
$\bullet$ {\bf The $\pi$-equations.}
Suppose that $\a,\b\in\Phi$ and $\angle(\a,\b)=\pi/2$.
Consider all pairs of roots $\{\g,\d\}$ such that
$\g=-\d$ and $\g,\d$ are not orthogonal to $\a$ and $\b$.
There are two possibilities: the first is $(\gamma,\a)=(\gamma,\b)$,
and then $(\d,\a)=(\d,\b)$. We may assume that
$(\g,\a)=(\g,\b)=2\pi/3$.
Put
$$
S_{\pi}(\a,\b)=\{(\g,\d)\mid \g+\d=0,\angle(\g,\a)=\angle(\g,\b)=2\pi/3\}.
$$
The second possibility is that one of the angles $\angle(\g,\a)$,
$\angle(\d,\a)$ is $2\pi/3$. We may assume that
$\angle(\g,\a)=2\pi/3$, and then $\angle(\g,\b)=\pi/3$,
$\angle(\d,\a)=\pi/3$, $\angle(\d,\b)=2\pi/3$.
Put
$$
S'_{\pi}(\a,\b)=\{(\g,\d)\mid \g+\d=0,\angle(\g,\a)=\angle(\d,\b)=2\pi/3\}.
$$

Consider the following equation on a vector
$v=(v_\l)_{\l\in\Lambda}\in V$:
\begin{equation}\label{pi-equation}
\sum_{s=1}^l\langle\a,\a_s\rangle \hat v_{s}\cdot\sum_{s=1}^l\langle\b,\a_s\rangle \hat{v}_{s} =
\sum_{(\g,\d)\in S'_{\pi}(\a,\b)}v_\g v_\d-\sum_{(\g,\d)\in S_{\pi}(\a,\b)}v_\g v_\d.
\end{equation}
We will call it {\it the $\pi$-equation} corresponding to the pair
$(\a,\b)$. 

Note that $|S_{\pi}(\a,\b)|=|S'_{\pi}(\a,\b)|=2(l-1),6,8,12$ for
$\Phi=\D_l,\E_6,\E_7,\E_8$, respectively.
As in the previous case, we can construct a maximal square
corresponding to $S_\pi(\a,\b)$.
For any pair $(\g,\d)\in S_\pi(\a,\b)$ we have
$\g+\a\in\Phi$ и $(\g+\a,\b)=(\g,\b)=-1/2$,
therefore $\g+\a+\b\in\Phi$. Moreover,
$(\g+\a+\b,\a)=(\g+\a+\b,\b)=(-\g,\a)=(-\g,\b)=1/2$
and $-\g+(\g+\a+\b)=\a+\b$.
This means that the roots $\{-\g\mid(\g,\d)\in S_\pi(\a,\b)\}$
together with $\a$, $\b$ form a maximal square.
It is easy to see that the roots $\{\g+\a,\b+\d\mid(\g,\d)\in
S'_\pi(\a,\b)\}$ together with
$\a$, $\b$ form (the same) maximal square.

As in the previous case, the constructed sets $S_\pi(\a,\b)$ and
$S'_\pi(\a,\b)$ are related to embeddings $\A_3\subset\Phi$:
if $(\g,\d)\in S_\pi(\a,\b)$, the roots $\a,\g,\b$ form a fundamental
system of a root subsystem $\Psi\subseteq\Phi$ of type $\A_3$.
We can write $(\g,\d)=(010,-010)$ in Dynkin notation with respect to
this fundamental system.
Moreover, $\Psi$ contains another pair of roots from $S_\pi(\a,\b)$,
namely, $(-111,111)$.
On the other hand, the pairs $(-110,110)$ and $(011,-011)$ are in
$S'(\a,\b)$.
The analogue of Lemma~\ref{2pi/3-equiv} holds in this situation:
\begin{lemma}\label{pi-equiv}
Suppose that $\a,\b\in\Phi$, $\a\perp\b$. Let
$\Omega=\{\b_1,\dots,\b_{-1}\}$ be a maximal square such that
$\b_1=\a$, $\b_{-1}=\b$, and $\b_i\perp\b_{-i}$ for every $i$.
Then
\begin{align*}
S_\pi(\a,\b)&=\{(-\b_i,\b_i)\mid i=2,\dots,-2\},\\
S'_\pi(\a,\b)&=\{(\b_i-\b_1,\b_1-\b_i)\mid i=2,\dots,-2\},
\end{align*}
and the $\pi$-equation corresponding to $(\a,\b)$ can be rewritten is
follows.
$$
\sum_{s=1}^l\langle\b_1,\a_s\rangle \hat v_{s}\cdot\sum_{s=1}^l\langle\b_{-1},\a_s\rangle \hat v_{s} =
\sum_{i\neq\pm 1}(v_{\b_1-\b_i}v_{\b_i-\b_1}-v_{-\b_i}v_{\b_i}).
$$
\end{lemma}

\par\smallskip
To reiterate, we get one $\pi/2$-equation for every maximal square,
one $2\pi/3$-equation and one $\pi$-equation for every maximal square
with a chosen pair of orthogonal roots in it.

\section{Preliminary lemmas}

We encountered embeddings $\A_3\subseteq\Phi$; we will use the fact
that every such embedding can be expanded to an embedding
$\D_4\subseteq\Phi$.
\begin{lemma}\label{lemma: A3 in D4}
Recall that $\Phi=\E_l$ or $\D_l$ ($l\geq 5$).
Every subsystem $\Psi\subseteq\Phi$ of type $\A_3$ can be embedded
into a subsystem of type $\D_4$. To be precise, if
$\a,\b,\g\in\Phi$ are roots such that $\a\perp\g$,
$\angle(\a,\b)=\angle(\b,\g)=2\pi/3$, then
there is a root $\d\in\Phi$ such that $\d\perp\a$, $\d\perp\g$, and
$\angle(\d,\b)=2\pi/3$.
\end{lemma}
\begin{proof}
In the case $\Phi=\E_l$ all subsystems of type $\A_3$ in $\Phi$ lie in
one orbit with respect to the action of the Weyl group $W(\E_l)$. This
follows, for example, from the tables in Carter's
paper~\cite{Carter_conjugacy}.
Therefore it remains to show the statement for a single subsystem of
type $\A_3$: for example, we may assume that $\a=\a_2$, $\b=\a_4$,
$\g=\a_3$ and take $\d=\a_5$.
In the case $\Phi=\D_l$ there are two orbits of subsystems of type
$\A_3$ with respect to the action of the Weyl group
$W(\D_l)$. This immediately follows from the computations
in~\cite[\S~9]{Carter_conjugacy}.
For one of the orbits we may assume that $\a=\a_{l-1}$, $\b=\a_{l-2}$,
$\g=\a_l$, and take $\d=\a_{l-3}$; for the other orbit we may assume
that $\a=\a_{l-3}$, $\b=\a_{l-2}$, $\g=\a_{l-1}$, and take $\d=\a_l$.
\end{proof}

Now we describe the possible relative positions of a root
$\rho\in\Phi$ and a maximal square $\Omega=\{\b_1,\dots,\b_{-1}\}$.
\begin{lemma}\label{lem:root_and_square_set}
Let $\Omega=\{\b_1,\dots,\b_{-1}\}$ be a maximal square, and let
$\rho\in\Phi$ be a root.
Exactly one of the following holds:
\par\smallskip
$(1)$ There exists $i$ such that $\rho=\b_i$,
$\angle(\rho,\b_{-i})=\pi/2$, and
$\angle(\rho,\b_j)=\pi/3$ for $j\neq\pm i$.
\par\smallskip
$(2)$ There exists $i$ such that $\rho=-\b_i$,
$\angle(\rho,\b_{-i})=\pi/2$, and
$\angle(\rho,\b_j)=2\pi/3$ for $j\neq\pm i$.
\par\smallskip
$(3)$ There exists $i$ such that $\rho\perp\b_i$ and
$\rho\perp\b_{-i}$; for every $j=1,\dots,-1$ either $\rho\perp\b_j$,
$\rho\perp\b_{-j}$, or one of the angles $\angle(\rho,\b_j)$,
$\angle(\rho,\b_{-j})$ equals $\pi/3$, while the other equals $2\pi/3$.
\par\smallskip
$(4)$ For every $i$ one of the angles $\angle(\rho,\b_i)$, $\angle(\rho,\b_{-i})$ equals $\pi/2$,
while the other equals $\pi/3$.
\par\smallskip
$(5)$ For every $i$ one of the angles $\angle(\rho,\b_i)$, $\angle(\rho,\b_{-i})$ equals $\pi/2$,
while the other equals $2\pi/3$.
\par\smallskip
Moreover, $(\rho,\b_i+\b_{-i})$ equals $1$, $-1$, $0$, $1/2$, $-1/2$
in cases $(1)$, $(2)$, $(3)$, $(4)$, $(5)$, respectively.
\end{lemma}
\begin{proof}
If $\rho\in\Omega$, we have $\rho=\b_i$ for some $i$.
By the definition of maximal square,
the angle between $\rho$ and every other root in $\Omega$ is equal to
$\pi/3$. In this case, $(1)$ holds.

If $-\rho\in\Omega$, we can apply the above observation to $-\rho$; in
this case, $(2)$ holds.

Now we assume that $\pm\rho\not\in\Omega$. Hence, $(\rho,\b_i)$ is
equal to $0$ or $\pm 1/2$. Suppose that there exists $i$ such that
$(\rho,\b_i) = 0$.
\begin{itemize}
\item
If there exists $i$ such that $(\rho,\b_i)=(\rho,\b_{-i})=0$, then for
every $j\neq i$ we have
$$
(\rho,\b_j)+(\rho,\b_{-j})=(\rho,\b_j+\b_{-j})=(\rho,\b_i+\b_{-i})=0.
$$
This means that either $(\rho,\b_j)=(\rho,\b_{-j})=0$, or one of these
scalar products equals $1/2$, and the other equals
$-1/2$. Therefore $(3)$ holds. Note that in this case there exists
$j$ such that $(\rho,\b_j)\neq 0$: otherwise $\rho$ would be
orthogonal to every root in $\Omega$, which is impossible.

\item
If there exists $i$ such that $(\rho,\b_i)=0$ and
$(\rho,\b_{-i})=1/2$, then for every $j\neq i$ we have
$(\rho,\b_j)+(\rho,\b_{-j})=1/2$. This means that one of these
products is equal to $0$, and the other is equal to $1/2$.
Therefore $(4)$ holds.

\item
Similarly, if there exists $i$ such that $(\rho,\b_i)=0$ and
$(\rho,\b_{-i})=-1/2$, then for every $j\neq i$ we have
$(\rho,\b_j)+(\rho,\b_{-j})=-1/2$. This means that one of these
products is equal to $0$, and the other is equal to $-1/2$.
Therefore $(4)$ holds.
\end{itemize}
Finally let us consider the remaining case: suppose that for every $i$
the scalar product $(\rho,\b_i)$ is not equal to $0$.
We must show that this is impossible.
If for some $i$ we have 
$(\rho,\b_i)=(\rho,\b_{-i})=1/2$,
then $\b_i+\b_{-i}-\rho$ is a root, and its sum with $\rho$ is
$\b_i+\b_{-i}$. Therefore $\rho\in\Omega$, and we are in the case
$(1)$.

On the other hand, if for some $i$ we have
$(\rho,\b_i)=(\rho,\b_{-i})=-1/2$, then $\rho+\b_i+\b_{-i}\in\Phi$ and
$-\rho+(\rho+\b_i+\b_{-i})=\b_i+\b_{-i}$. Therefore $-\rho\in\Omega$,
and we are in the case $(2)$.

Finally, we can choose $i$ such that $(\rho,\b_i)=1/2$,
$(\rho,\b_{-i})=-1/2$.
Let us show that we are in the case $(3)$.
The roots $-\b_i,\rho,\b_{-i}$ span a root subsystem
$\Psi\subseteq\Phi$ of type $\A_3$. By Lemma~\ref{lemma: A3 in D4}, we
can embed it into a root subsystem of type $\D_4$.
Therefore there exists $\s\in\Phi$ such that
$\s\perp\b_i$, $\s\perp\b_{-i}$, and $(\s,\rho)=-1/2$. But
$-\s-\rho+\b_i,\s+\rho+\b_{-i}\in\Phi$, the sum of these two roots is
$\b_i+\b_{-i}$, and both of them are orthogonal to $\rho$. This means
that we are in the case $(3)$.
\end{proof}

\begin{definition}
Let $\Omega$ be a maximal square in $\Phi$, and let $\rho\in\Phi$ be
a root.
We say that the angle between $\rho$ and $\Omega$ is equal to
$\angle(\rho,\Omega) = 0$, $\pi$, $\pi/2$, $\pi/3$, $2\pi/3$, if in
the Lemma~\ref{lem:root_and_square_set} the condition
$(1)$, $(2)$, $(3)$, $(4)$, $(5)$ holds, respectively.
\end{definition}

\begin{lemma}\label{lem:modified_square}
Let $\Omega=\{\b_1,\dots,\b_{-1}\}$ be a maximal square.
Suppose that $j\in\{1,\dots,-1\}$.
Put $\g_i=\b_j-\b_i$ for every $i\neq\pm j$, $\g_j=\b_j$,
$\g_{-j}=-\b_{-j}$.
$\Omega'=\{\g_1,\dots,\g_{-1}\}$ is a maximal square.
\end{lemma}
\begin{proof}
An easy calculation shows that $(\g_i,\g_{-i})=0$ for all $i$, and
$(\g_i,\g_t)=1/2$ for all $i\neq t$.
\end{proof}

\section{Action of the elementary subgroup}

Suppose $\rho\in\Phi$, $\xi\in R$. We work with the adjoint
representation, therefore the action of the elementary root unipotent
$x_\rho(\xi)$ on the basis of $V$ is described by the following lemma.

\begin{lemma}[Matsumoto]
\begin{enumerate}
\item If $\lambda\in\Phi$, $\lambda+\rho\notin\Phi\cup\{0\}$, then $x_\rho(\xi)e^\l=e^\l$;
\item if $\lambda,\lambda+\rho\in\Phi$, then $x_\rho(\xi)e^\l=e^\l+N_{\rho,\l}\xi e^{\l+\rho}$;
\item $x_\rho(\xi)\hat{e}^s=\hat{e}^s-\la\rho,\a_s\ra\xi e^\rho$ for $s=1,\dots,l$;
\item $x_\rho(\xi)e^{-\rho}=e^{-\rho}+\sum_{s=1}^l m_s(\rho)\xi\hat{e}^s-\xi^2e^\rho$.
\end{enumerate}
\end{lemma}
\begin{proof}
See~\cite[Lemma 2.3]{Matsumoto}.
\end{proof}

This immediately implies the following description of the action of
$x_\rho(\xi)$ on the coordinates of $v=(v_\lambda)\in V$. We will use it
without any further reference.

\begin{lemma}\label{action_on_vector}
\begin{enumerate}
\item If $\lambda\in\Phi$, $\lambda-\rho\notin\Phi\cup\{0\}$, then $(x_\rho(\xi)v)_\l=v_\l$;
\item if $\lambda,\lambda-\rho\in\Phi$, then $(x_\rho(\xi)v)_\l=v_\l+N_{\rho,\l-\rho}\xi v_{\l-\rho}$;
\item $\widehat{(x_\rho(\xi)v)}_s=\hat{v}_s+m_s(\rho)\xi v_{-\rho}$.
\item $(x_\rho(\xi)v)_\rho=v_\rho-\sum_{s=1}^l\la\rho,\a_s\ra\xi \hat{v}_s-\xi^2v_{-\rho}$.
\end{enumerate}
In particular, if $\angle(\rho,\l)=\pi/2$, $2\pi/3$ or $\pi$, then
$(x_\rho(\xi)v)_\l=v_\l$.
\end{lemma}

We will often use the following description of the action of
$x_\rho(\xi)$ on the zero weights.

\begin{lemma}\label{action_on_zero_weights}
Suppose that $\b,\rho\in\Phi$, $\xi\in R$, $v\in V$,
$w=x_\rho(\xi)v$. Then
$\sum_s\la\b,\a_s\ra\hat{w}_s
=\sum_s\la\b,\a_s\ra\hat{v}_s
+\xi\la\b,\rho\ra v_{-\rho}$.
\end{lemma}
\begin{proof}
\begin{align*}
\sum_s\la\b,\a_s\ra\hat{w}_s&=\sum_s\la\b,\a_s\ra(\hat{v}_s+m_s(\rho)\xi v_{-\rho})\\
&=\sum_s\la\b,\a_s\ra\hat{v}_s+\xi\sum_s\la\b,m_s(\rho)\a_s\ra v_{-\rho}\\
&=\sum_s\la\b,\a_s\ra\hat{v}_s+\xi\la\b,\rho\ra v_{-\rho}
\end{align*}
\end{proof}

Suppose that $v\in V$, $\a,\b\in\Phi$, and
$\Omega=\{\b_1,\dots,\b_{-1}\}$
is a maximal square such that
$\b_1=\a$, $\b_{-1}=\b$, and $\b_i\perp\b_{-i}$ for every $i$.
We need the following notation for the polynomials in the equations
$(\ref{pi/2-equation})$,
$(\ref{2pi/3-equation})$,
$(\ref{pi-equation})$:
\begin{align*}
f^{\pi/2}_{\a,\b}(v)&=v_\a v_\b-\sum_{\{\g,\d\}\in S_{\pi/2}(\a,\b)}N_{\a,-\g}N_{\b,-\d}v_\g v_\d,\\
f^{2\pi/3}_{\a,\b}(v)&=\sum_{i\neq\pm 1}N_{\b_1,-\b_i}v_{\b_1-\b_i}v_{\b_i}-v_{\b_1}\sum_{s=1}^l\la\b_{-1},\a_s\ra\hat{v}_s,\\
f^{\pi}_{\a,\b}(v)&=\sum_{i\neq\pm 1}(v_{\b_1-\b_i}v_{\b_i-\b_1}-v_{-\b_i}v_{\b_i})-
\sum_{s=1}^l\la\b_1,\a_s\ra\hat{v}_s\cdot\sum_{s=1}^l\la\b_{-1},\a_s\ra\hat{v}_s.
\end{align*}

\begin{proposition}\label{key_proposition}
Let $\a,\b,\rho\in\Phi$ be roots such that $\a\perp\b$, and let
$v\in V$ be a vector.
 Take $\xi\in R$ and put $w=x_\rho(\xi)v$.
Suppose that $\ph\in\{\pi/2,2\pi/3,\pi\}$.
Then $f^\ph_{\a,\b}(w)$ is a linear combination of polynomials of the
form $f^\psi_{\g,\d}(v)$.
\end{proposition}

We shall prove Proposition~\ref{key_proposition} in the next
section. Now we can derive our main result from it.

\begin{theorem}\label{thm:main}
The set of vectors $v\in V$ satisfying the equations
$(\ref{pi/2-equation})$,
$(\ref{2pi/3-equation})$,
$(\ref{pi-equation})$ for all $\a,\b\in\Phi$, $\a\perp\b$,
is invariant under the action of the group $E(\Phi,R)$.
\end{theorem}
\begin{proof}
It suffices to prove that if $v\in V$ satisfies the above equations,
then $w=x_\rho(\xi)v$ satisfies them for every $\rho\in\Phi$, $\xi\in
R$. Indeed, by Proposition~\ref{key_proposition}, each of the
polynomials $f^\ph_{\a,\b}(w)$ is equal to a linear combination of
these polynomials applied to $v$, which is zero.
\end{proof}

\begin{corollary}
If $v\in V$ is a column of an element $g\in E(\Phi,R)$ corresponding
to any root $\rho\in\Lambda^*=\Phi$, then $v$ satisfies the equations
$(\ref{pi/2-equation})$,
$(\ref{2pi/3-equation})$,
$(\ref{pi-equation})$ for all $\a,\b\in\Phi$, $\a\perp\b$.
\end{corollary}
\begin{proof}
We have $v=ge^\rho$. It is obvious that $e^\rho$ satisfies those
equations, so by Theorem~\ref{thm:main} $v$ satisfies them too.
\end{proof}

\section{Proof of Proposition~\ref{key_proposition}}

Let $\Omega=\{\b_1,\dots,\b_{-1}\}$
is a maximal square such that
$\b_1=\a$, $\b_{-1}=\b$, and $\b_i\perp\b_{-i}$ for every $i$. 
We shall explore the five cases described in
Lemma~\ref{lem:root_and_square_set}.

\begin{enumerate}
\item
Suppose that $(\rho,\Omega) = 0$. This means that $\rho=\beta_j$ for
some $j$.
\begin{itemize}
\item {\bf The $\pi/2$-equation.}
The discussion following Definition~\ref{def:maximal_square} shows
that the $\pi/2$-equation depends only on a maximal square and not on
the choice of orthogonal roots $\a$, $\b$ in it. Thus we may assume
that $j=1$.
Then
\begin{align*}
& f^{\pi/2}_{\a,\b}(w)-f^{\pi/2}_{\a,\b}(v) \\
&=
\left( -\sum_{s=1}^l\la\b_1,\a_s\ra\xi\hat{v}_s - \xi^2
  v_{-\b_1}\right)v_{\b_{-1}}\\
& - \xi\sum_{i\neq\pm 1}N_{\b_1,-\b_i}N_{\b_{-1},-\b_{-i}}
N_{\b_1,\b_i-\b_1}v_{\b_i-\b_1}v_{\b_{-i}}\\
& - \xi^2\sum_{i\geq 2}N_{\b_1,-\b_i}N_{\b_{-1},-\b_{-i}}
N_{\b_1,\b_i-\b_1}N_{\b_1,\b_{-i}-\b_1} v_{\b_i-\b_1}v_{\b_{-i}-\b_1}.
\end{align*}
First, note that
$N_{\b_1,\b_i-\b_1}=-N_{\b_1,-\b_i}$.
Moreover,
$N_{\b_{-1},-\b_{-i}} = N_{\b_{-1},\b_1-\b_i}$ and
$N_{\b_1,\b_{-i}-\b_1} = -N_{-\b_1,\b_1-\b_{-i}}$.
Therefore the terms on the right-hand side containing $\xi^2$ sum up
to $-\xi^2 f^{\pi/2}_{\b_{-1},-\b_1}(v)$.
The rest sums up to $-\xi f^{2\pi/3}_{\b_{-1},-\b_1}(v)$. Indeed,
$N_{\b_1,-\b_i} = - N_{\b_1,\b_i-\b_1}$, and (by
Lemma~\ref{lem:modified_square}) the roots $\b_i-\b_1$ together with
$\b_{-1}$ and $-\b_1$ form a maximal square.
Hence
$$
f^{\pi/2}_{\a,\b}(w)-f^{\pi/2}_{\a,\b}(v)
= -\xi f^{2\pi/3}_{\b_{-1},-\b_1}(v) - \xi^2 f^{\pi/2}_{\b_{-1},-\b_1}(v).
$$
\item {\bf The $2\pi/3$-equation.}
First suppose that $j\neq\pm 1$.
Then $w_{\b_1-\b_i} = v_{\b_1-\b_i}$ for $i\neq -j$.
Thus
\begin{align*}
f^{2\pi/3}_{\a,\b}(w)-f^{2\pi/3}_{\a,\b}(v) &=
\sum_{\substack{i\neq\pm 1\\i\neq\pm j}}
\left(N_{\b_1,-\b_i}v_{\b_1-\b_i}\xi
N_{\b_j,\b_i-\b_j}v_{\b_i-\b_j}\right)\\
&+N_{\b_1,-\b_{-j}}\xi N_{\b_j,-\b_{-1}}v_{-\b_{-1}}v_{\b_{-j}}\\
&- N_{\b_1,-\b_j} v_{\b_1-\b_j}\sum_{s=1}^l\la\b_j,\a_s\ra\xi\hat{v}_s\\
&- N_{\b_1,-\b_j} v_{\b_1-\b_j}\xi^2v_{-\b_j}\\
&-\xi N_{\b_j,\b_1-\b_j}v_{\b_1-\b_j}
\sum_{s=1}^l\la\b_{-1},\a_s\ra\hat{v}_s\\
&-v_{\b_1}\xi\la\b_{-1},\b_j\ra v_{-\b_j}\\
&-\xi^2 N_{\b_j,\b_1-\b_j}v_{\b_1-\b_j}\la\b_{-1},\b_j\ra v_{-\b_j}.
\end{align*}
Note that $\la\b_{-1},\b_j\ra=1$ and $N_{\b_j,\b_1-\b_j} = -
N_{\b_1,-\b_j}$, so the terms containing $\xi^2$ cancel each other out.
Using the cocycle identity
$$N_{\b_1,-\b_i}N_{\b_j-\b_1,\b_1-\b_i} =
N_{\b_j,-\b_i}N_{\b_j-\b_1,\b_1}$$
and Lemma~\ref{lem:modified_square},
it is easy to show that the rest
yields
$$
f^{2\pi/3}_{\a,\b}(w)-f^{2\pi/3}_{\a,\b}(v) = \xi
N_{\b_1,-\b_j}f^{2\pi/3}_{\b_1-\b_j,\b_1-\b_{-j}}(v).
$$

Now suppose that $j=1$. By Lemma 9 we have
$\sum_{s=1}^l\la\b_{-1},\a_s\ra\hat{w}_s =
\sum_{s=1}^l\la\b_{-1},\a_s\ra\hat{v}_s$, since $\la\b_{-1},\b_1\ra =
0$. Therefore
\begin{align*}
f^{2\pi/3}_{\a,\b}(w) - f^{2\pi/3}_{\a,\b}(v) &=
\sum_{i\neq\pm 1} N_{\b_1,-\b_i}\xi N_{\b_1,-\b_i}v_{-\b_i}v_{\b_i}\\
&+ \sum_{i\neq\pm 1} N_{\b_1,-\b_i}v_{\b_1-\b_i}\xi
N_{\b_1,\b_i-\b_1}v_{\b_i-\b_1}\\
&+ \sum_{i\neq\pm 1} N_{\b_1,-\b_i}\xi^2
N_{\b_1,-\b_i}N_{\b_1,\b_i-\b_1}v_{-\b_i}v_{\b_i-\b_1}\\
&+ \left(\sum_{s=1}^l\la\b_1,\a_s\ra\xi\hat{v}_s+\xi^2v_{-\b_1}\right)
\sum_{s=1}^l\la\b_{-1},\a_s\ra\hat{v}_s.
\end{align*}
Note that $N_{\b_1,\b_i-\b_1} = -N_{\b_1,-\b_i}$.
It is easy to see that
$$
f^{2\pi/3}_{\a,\b}(w) - f^{2\pi/3}_{\a,\b}(v) 
= -\xi f^{\pi}_{\a,\b}(v) + \xi^2 f^{2\pi/3}_{-\b_1,-\b_{-1}}(v).
$$

Finally, suppose that $j=-1$. In this case
$w_{\b_1-\b_i}=v_{\b_1-\b_i}$ for every $i\neq\pm 1$, and
$w_{\b_1}=v_{\b_1}$. We obtain
\begin{align*}
f^{2\pi/3}_{\a,\b}(w) - f^{2\pi/3}_{\a,\b}(v) &=
\sum_{i\neq\pm 1}N_{\b_1,-\b_i}v_{\b_1-\b_i}\xi
N_{\b_{-1},\b_i-\b_{-1}}v_{\b_i-\b_{-1}}\\
&-v_{\b_1}\xi\la\b_{-1},\b_{-1}\ra v_{-\b_{-1}}\\
&= -2\xi f^{\pi/2}_{\b_1,-\b_{-1}}(v).
\end{align*}
\item {\bf The $\pi$-equation.} 
First suppose that $j\neq\pm 1$. Then $w_{\b_1-\b_i}=v_{\b_1-\b_i}$
for $i\neq -j$, and $w_{\b_i-\b_1}=v_{\b_i-\b_1}$ for $i\neq
j$. Moreover, $w_{-\b_i}=v_{-\b_i}$ for $i\neq -j$. Therefore
\begin{align*}
f^{\pi}_{\a,\b}(w) - f^{\pi}_{\a,\b}(v) &=
-\sum_{\substack{i\neq\pm 1\\i\neq\pm j}}v_{-\b_i}\xi
N_{\b_j,\b_i-\b_j}v_{\b_i-\b_j}\\
& + v_{\b_1-\b_j}\xi N_{\b_j,-\b_1}v_{-\b_1}\\
& + v_{-\b_j}\left(\xi^2 v_{-\b_j} +
  \sum_{s=1}^l\la\b_j,\a_s\ra\xi\hat{v}_s\right)\\
& + \xi N_{\b_j,-\b_{-1}}v_{-\b_{-1}}v_{\b_{-j}-\b_1}\\
& - \xi N_{\b_j,-\b_j-\b_{-j}}v_{-\b_j-\b_{-j}}v_{\b_{-j}}\\
& - \sum_{s=1}^l\la\b_1,\a_s\ra\hat{v}_s\xi v_{-\b_j}\\
& - \sum_{s=1}^l\la\b_{-1},\a_s\ra\hat{v}_s\xi v_{-\b_j}\\
& - \xi^2v_{-\b_j}v_{-\b_j}.
\end{align*}
The terms containing $\xi^2$ cancel each other out. Arguing as above,
it is not hard to see that
$$
f^{\pi}_{\a,\b}(w) - f^{\pi}_{\a,\b}(v) = \xi f^{2\pi/3}_{-\b_j,\b_{-j}}(v).
$$

For $j=1$ we have $w_{\b_i-\b_1}=v_{\b_i-\b_1}$ and
$w_{-\b_i}=v_{-\b_i}$ for all $i\neq\pm 1$. Moreover, by
Lemma~\ref{action_on_zero_weights}, we have
$$\sum_{s=1}^l\la\b_{-1},\a_s\ra\hat{w}_s
 = \sum_{s=1}^l\la\b_{-1},\a_s\ra\hat{v}_s,$$
since $\la\b_{-1},\b_1\ra=0$.

Thus
\begin{align*}
f^{\pi}_{\a,\b}(w) - f^{\pi}_{\a,\b}(v) &=
\sum_{i\neq\pm 1}\xi N_{\b_1,-\b_i}v_{-\b_i}v_{\b_i-\b_1}\\
&- \sum_{i\neq\pm 1}v_{-\b_i}\xi N_{\b_1,\b_i-\b_1}v_{\b_i-\b_1}\\
&- \xi\la\b_1,\b_1\ra
v_{-\b_1}\cdot\sum_{s=1}^l\la\b_{-1},\a_s\ra\hat{v}_s\\
& = -2\xi f^{2\pi/3}_{-\b_1,-\b_{-1}}(v).
\end{align*}

Finally, suppose that $j=-1$. Then $w_{\b_1-\b_i}=v_{\b_1-\b_i}$ and
$w_{-\b_i}=v_{-\b_i}$ for all $i\neq\pm 1$. Moreover, by
Lemma~\ref{action_on_zero_weights}, we have
$\sum_{s=1}^l\la\b_1,\a_s\ra\hat{w}_s
 = \sum_{s=1}^l\la\b_1,\a_s\ra\hat{v}_s$, since
 $\la\b_1,\b_{-1}\ra=0$.
Thus
\begin{align*}
f^{\pi}_{\a,\b}(w) - f^{\pi}_{\a,\b}(v) &=
\sum_{i\neq\pm 1}v_{\b_1-\b_i}\xi N_{\b_{-1},-\b_{-i}}v_{-\b_{-i}}\\
&-\sum_{i\neq\pm 1}v_{-\b_i}\xi
N_{\b_{-1},\b_i-\b_{-1}}v_{\b_i-\b_{-1}}\\
&- \sum_{s=1}^l\la\b_1,\a_s\ra\hat{v}_s \cdot \xi\la\b_{-1},\b_{-1}\ra
v_{-\b_{-1}} \\
&= -2\xi f^{2\pi/3}_{-\b_{-1},-\b_1}(v).
\end{align*}

\end{itemize}
\item
Suppose that $(\rho,\Omega) = \pi$. This means that $\rho=-\beta_j$ for
some $j$.
\begin{itemize}
\item {\bf The $\pi/2$-equation.} 
Note that $\beta_i-\rho=\beta_i+\beta_j$ is never a root, hence by
Lemma~\ref{action_on_vector} we have
$w_{\beta_i} = v_{\beta_i}$ for all $i$, and $f^{\pi/2}_{\a,\b}(w) =
f^{\pi/2}_{\a,\b}(v)$.
\item {\bf The $2\pi/3$-equation.}
Here we have $w_{\beta_i}=v_{\beta_i}$, $w_{\beta_1}=v_{\beta_1}$.

If $j\neq\pm 1$, then $w_{\beta_1-\beta_i}=v_{\beta_1-\beta_i}$ for
$i\neq j$. Using Lemma~\ref{action_on_zero_weights}, we get
$$
f^{2\pi/3}_{\a,\b}(w) = 
f^{2\pi/3}_{\a,\b}(v) +
N_{\beta_1,-\beta_j}N_{-\beta_j,\beta_1}\xi v_{\beta_1}v_{\beta_j} -
v_{\beta_1}\xi\la\beta_{-1},-\beta_j\ra v_{\beta_j}.$$
It remains to note that $N_{\beta_1,-\beta_j}=-N_{-\beta_j,\beta_1}$
and $\la\beta_{-1},-\beta_j\ra=-1$, so that $f^{2\pi/3}_{\a,\b}(w) = 
f^{2\pi/3}_{\a,\b}(v)$.

For $j=1$ we have $w_{\beta_1-\beta_i}=v_{\beta_1-\beta_i}$ for all
$i$, and (by Lemma~\ref{action_on_zero_weights})
$$\sum_{s=1}^l\la\beta_{-1},\alpha_s\ra\hat{w}_s
- \sum_{s=1}^l\la\beta_{-1},\alpha_s\ra\hat{v}_s =
\xi\la\beta_{-1},-\beta_1\ra v_{\beta_1} = 0,$$ so that
$f^{2\pi/3}_{\a,\b}(w) = f^{2\pi/3}_{\a,\b}(v)$ again.

Finally, if $j=-1$, then
$$
w_{\beta_1-\beta_i} = v_{\beta_1-\beta_i} +
N_{-\beta_{-1},\beta_{-i}}\xi v_{\beta_{-i}}$$
and $$\sum_{s=1}^l\la\beta_{-1},\alpha_s\ra\hat{w}_s
- \sum_{s=1}^l\la\beta_{-1},\alpha_s\ra\hat{v}_s =
\xi\la\beta_{-1},-\beta_{-1}\ra v_{\beta_{-1}} = -2\xi
v_{\beta_{-1}}.$$
Therefore
\begin{align*}
f^{2\pi/3}_{\a,\b}(w) - f^{2\pi/3}_{\a,\b}(v) &=
\sum_{i\neq\pm 1}N_{\beta_1,-\beta_i}N_{-\beta_{-1},\beta_{-i}}\xi
v_{\beta_{-i}}v_{\beta_i} + 2\xi v_{\beta_1}v_{\beta_{-1}}\\
&= -\sum_{i\geq 2}2\xi N_{\beta_1,-\beta_i}N_{\beta_{-1},-\beta_{-i}} +
2\xi v_{\beta_1}v_{\beta_{-1}} \\
&= 2\xi f^{\pi/2}_{\beta_1,\beta_{-1}}(v).
\end{align*}
\item {\bf The $\pi$-equation.} 
First suppose that $j\neq\pm 1$. Then $w_{\b_1-\b_i}=v_{\b_1-\b_i}$
for all $i\neq j$, and $w_{\b_i-\b_1}=v_{\b_i-\b_1}$ for all $i\neq
-j$. Moreover, $w_{\b_i}=v_{\b_i}$ for all $i$.
This means that
\begin{align*}
f^{\pi}_{\alpha,\beta}(w)-f^{\pi}_{\a,\b}(v) &=
\sum_{\substack{i\neq\pm 1\\i\neq\pm j}}(-\xi
N_{-\b_j,\b_j-\b_i}v_{\b_j-\b_i}v_{\b_i})\\
&+ \xi N_{-\b_j,\b_1}v_{\b_1}v_{\b_j-\b_1}\\
&+ v_{\b_1-\b_{-j}}\xi N_{-\b_j,\b_{-1}}v_{\b_{-1}}\\
&-\left(-\sum_{s=1}^l\la -\b_j,\a_s\ra\xi\hat{v}_s-\xi^2
  v_{\b_j}\right)v_{\b_j}\\
& -\sum_{s=1}^l\la\beta_1,\a_s\ra\hat{v}_s\xi\la\b_{-1},-\b_j\ra
v_{\b_j}\\
& -\sum_{s=1}^l\la\beta_{-1},\a_s\ra\hat{v}_s\xi\la\b_1,-\b_j\ra
v_{\b_j}\\
& -\xi^2v_{\b_j}v_{\b_j}.
\end{align*}
The last four lines sum up to
$$-\sum_{s=1}^l\la -\b_1-b_{-1}+\b_j,\a_s\ra\xi\hat{v}_s v_{\beta_j}
 = -\xi v_{\beta_j}\sum_{s=1}^l\la
 -\b_{-j},\a_s\ra\hat{v}_s.$$
Applying Lemma~\ref{lem:modified_square}
 and noticing that $N_{-\b_j,\b_j-\b_i}=-N_{\b_j,\b_i-\b_j}$ we
 finally obtain
$$
f^{\pi}_{\alpha,\beta}(w)-f^{\pi}_{\a,\b}(v) = \xi f^{2\pi/3}_{\b_j,-\b_{-j}}(v).
$$

Next, suppose that $j=1$. Arguing exactly like in case (1), we get
$$
f^{\pi}_{\a,\b}(w) - f^\pi_{\a,\b}(v) =
2\xi f^{2\pi/3}_{\b_1,-\b_{-1}}(v).
$$
Similarly, for $j=-1$,
$$
f^{\pi}_{\a,\b}(w) - f^\pi_{\a,\b}(v) =
-2\xi f^{2\pi/3}_{\b_{-1},\b_1}(v).
$$
\end{itemize}
\item
Suppose that $\angle(\rho,\Omega) = \pi/2$. This means that for
some $j$ we have $(\rho,\beta_j)=1/2$ and $(\rho,\beta_{-j})=-1/2$ (or
vice versa). Note that
$\beta_j-\rho$ and $\beta_{-j}+\rho$ are orthogonal roots with sum
$\beta_{-j}+\beta{j}$; therefore they lie in $\Omega$.
By Chevalley commutator's formula we have
$x_{\rho}(\xi) = [x_{\beta_{-j}+\rho}(\xi),x_{-\beta_{-j}}(\pm 1)]$.
Thus we reduce the question to two previous cases, since
$\angle(\beta_{-j}+\rho,\Omega)=0$ and $\angle(-\beta_{-j},\Omega)=\pi$.
\item
Suppose that $(\rho,\Omega) = 2\pi/3$. This means that for every $i$
one of the scalar products $(\rho,\b_i)$, $(\rho,\b_{-i})$ equals $0$,
while the other equals $-1/2$.
\begin{itemize}
\item {\bf The $\pi/2$-equation.}
Note that $(\b_i,\rho)\leq 0$ for every $i$. Hence $w_{\b_i}=v_{\b_i}$
for every $i$ and we obtain
$f^{\pi/2}_{\a,\b}(w) = f^{\pi/2}_{\a,\b}(v)$.
\item {\bf The $2\pi/3$-equation.}
As above, we have $w_{\b_i}=v_{\b_i}$ for every $i$.

If $(\b_1,\rho)=-1/2$, then $w_{\b_1-\b_i} = v_{\b_1-\b_i}$ for every
$i\neq\pm 1$, and $\la\b_{-1},\rho\ra=0$, so that
$\sum_{s=1}^l\la\b_{-1},\a_s\ra\hat{w}_s
= \sum_{s=1}^l\la\b_{-1},\a_s\ra\hat{v}_s$. It follows that
$f^{2\pi/3}_{\a,\b}(w) = f^{2\pi/3}_{\a,\b}(v)$.

Now we assume that $(\b_1,\rho)=0$ and $(\b_{-1},\rho)=-1/2$.
Then $w_{\b_1-\b_i}=v_{\b_1-\b_i}$ whenever $(\b_i,\rho)=0$.
We obtain
\begin{align*}
f^{2\pi/3}_{\a,\b}(w) - f^{2\pi/3}_{\a,\b}(v) &=
\sum_{i\colon (\b_i,\rho)=-1/2}N_{\b_1,-\b_i} \xi
N_{\rho,\b_1-\b_i-\rho}v_{\b_1-\b_i-\rho}v_{\b_i}\\
&+\xi v_{\b_1}v_{-\rho}
\end{align*}
Note that exactly half of $2k-2$ indices $i=2,\dots,-2$ satisfy the
condition $(\b_i,\rho)=-1/2$, and for each one of them we have
$(\b_1-\b_i-\rho)+\b_i = \b_1-\rho$. Hence the roots
$\b_1,-\rho,\{(\b_1-\b_i-\rho),\b_i\}_{i\colon (\b_i,\rho)=-1/2}$ form
a maximal square.
Therefore
$$
f^{2\pi/3}_{\a,\b}(w) - f^{2\pi/3}_{\a,\b}(v) =
\xi f^{\pi/2}_{\b_1,-\rho}(v).
$$
\item {\bf The $\pi$-equation.}
As above, note that for exactly half of the indices $i=1,\dots,-1$
we have $(\b_i,\rho)=-1/2$, and for the other half we have
$(\b_i,\rho)=0$.
Put $J=\{i\mid (\b_i,\rho)=-1/2\}$, $K=\{i\mid (\b_i,\rho)=0\}$.
We know that $i\in J$ if and only if $-i\in K$.
Again, we have $w_{\b_i}=v_{\b_i}$ for all $i$,
and $w_{-\b_i}=v_{-\b_i}$ for $i\in K$.

First suppose that $1\in J$. Then
$w_{\b_1-\b_i}=v_{\b_1-\b_i}$ for all $i$, and
$w_{\b_i-\b_1}=v_{\b_i-\b_1}$ for $i\in J$.
In this case we have
\begin{align*}
f^{\pi}_{\a,\b}(w) - f^{\pi}_{\a,\b}(v) &=
\sum_{i\in K\setminus\{-1\}}v_{\b_1-\b_i}\xi N_{\rho,\b_i-\b_1-\rho}v_{\b_i-\b_1-\rho}\\
&- \sum_{i\in J\setminus\{1\}}\xi N_{\rho,-\b_i-\rho}v_{-\b_i-\rho}v_{\b_i}\\
&+ \xi v_{-\rho}\cdot\sum_{s=1}^l\la\b_{-1},\a_s\ra\hat{v}_s
\end{align*}
It is easy to see that the roots $\b_{-1}$, $-\rho$,
$\{\b_i-\b_1-\rho,\b_{-i}\}_{i\in K\setminus\{-1\}}$ form a maximal square.
It follow that
$$
f^{\pi}_{\a,\b}(w) - f^{\pi}_{\a,\b}(v) =
- \xi f^{2\pi/3}_{-\rho,\b_{-1}}(v).
$$

Finally, suppose that $1\in K$. Then
$w_{\b_1-\b_i} = v_{\b_1-\b_i}$ for $i\in K$, and
$w_{\b_i-\b_1} = v_{\b_i-\b_1}$ for all $i$.
Similarly,
\begin{align*}
f^{\pi}_{\a,\b}(w) - f^{\pi}_{\a,\b}(v) &=
\sum_{i\in J\setminus\{-1\}}\xi N_{\rho,\b_1-\b_i-\rho}v_{\b_1-\b_i-\rho} v_{\b_i-\b_1}\\
&- \sum_{i\in J\setminus\{-1\}}\xi N_{\rho,-\b_i-\rho}v_{-\b_i-\rho}v_{\b_i}\\
&+ \sum_{s=1}^l\la\b_1,\a_s\ra\hat{v}_s\cdot \xi v_{-\rho}
\end{align*}
It is easy to see that the roots $\b_1$, $-\rho$,
$\{\b_1-\b_i-\rho,\b_i\}_{i\in J\setminus\{-1\}}$ form a maximal
square.
It follows that
$$
f^{\pi}_{\a,\b}(w) - f^{\pi}_{\a,\b}(v) =
- \xi f^{2\pi/3}_{-\rho,\b_1}(v).
$$
\end{itemize}
\item
Suppose that $(\rho,\Omega) = \pi/3$.
This means that $(\rho,\beta_1)=0$ and $(\rho,\beta_{-1}) =
\pi/3$ (or vice versa). Then $\rho-\beta_{-1}$ is a root and by
Chevalley commutator's formula we have
$x_{\rho}(\xi) = [x_{\rho-\beta_{-1}}(\xi),x_{\beta_{-1}}(\pm 1)]$.
Thus we reduce the problem to previously discussed cases, since
$\angle(\rho-\beta_{-1},\Omega)=2\pi/3$ and
$\angle(\beta_{-1},\Omega)=0$.

\end{enumerate}

\bibliography{adjointeq}

\end{document}